\documentclass[11pt]{amsart}

\usepackage{graphicx}
\usepackage{amssymb}
\usepackage{epstopdf}
\newtheorem{theorem}{Theorem}[section]
\newtheorem{lemma}[theorem]{Lemma}
\newtheorem{corollary}[theorem]{Corollary}

\numberwithin{equation}{section}

\title[Min-max-min problems on
the unit circle]{Inverse Bernstein inequalities and min-max-min problems on
the unit circle
}
\author{ Tam\'as Erd\'elyi}
\address{Department of Mathematics, Texas A\&M University,
College Station, Texas 77843, USA
}
\email{terdelyi@math.tamu.edu}

\author{Douglas P. Hardin$^*$}
\address{Center for Constructive Approximation, Department of Mathematics, \hspace*{.1in}
Vanderbilt University,
Nashville, TN 37240, USA  }

 \email{doug.hardin@vanderbilt.edu}

\author{Edward B. Saff$^*$}
\email{edward.b.saff@vanderbilt.edu}

\thanks{\noindent $^*$ The research of these authors was supported, in part,
by the U. S. National Science Foundation under grant  DMS-1109266.  
}

\date{July 15, 2013}                                           

\subjclass[2000]{11C08, 41A17, 30C15, 52A40.}
\keywords{Chebyshev constants, polarization inequalities, Riesz energy, potentials}

\begin{document}
\begin{abstract}
 We give a short and elementary proof of an inverse
Bernstein-type inequality found by S. Khrushchev for the derivative of a polynomial having all its zeros on the unit circle . The inequality is used to show that equally-spaced points solve a min-max-min problem
for the logarithmic potential of such polynomials. Using techniques recently developed for polarization (Chebyshev-type)
problems, we show that this  optimality also holds for a large class of potentials, including the Riesz potentials $1/r^s$ with
$s>0.$
\end{abstract}

\maketitle
\section{Inverse Bernstein-type inequality}

Inequalities involving the derivatives of polynomials often occur in
approximation theory (see, e.g. \cite{BE-95}, \cite{DL-93}).   One of the most familiar of these inequalities is due to Bernstein which provides an upper bound for the derivative of a polynomial on the unit circle
 $\mathbb{T}$   of the complex plane.
In \cite{K-09}, S. Khrushchev derived a rather striking inverse Bernstein-type
inequality, a slight improvement of which may be stated as follows.

\begin{theorem}\label{Thm1.1}
If
$$P(z) = \prod_{j=1}^n{(z-z_j)}\,, \qquad z_j \in \mathbb{T}\,,$$
$$z_j = e^{it_j}, \qquad 0 \leq t_1 < t_2 < \cdots < t_n < 2\pi\,, \quad
t_{n+1} := t_1 + 2\pi\,,$$
\begin{equation}\label{mdef}
m := \min_{1 \leq j \leq n}\left(\max_{t \in [t_j,t_{j+1}]}{|P(e^{it})|}\right),\,
\end{equation}
then
\begin{equation}\label{main}
|P^{\prime}(z)|^2 \geq
\left( \frac n2 \right)^2 \left( |P(z)|^2 + (m^2 - |P(z)|^2)_+ \right)
\geq \left( \frac {nm}{2} \right)^2\,, \qquad z \in \mathbb{T}\,,
\end{equation}
where   $(x)_+ := \max\{x,0\}$.
\end{theorem}

Khrushchev used a potential theoretic method to prove his inequality.
Here we offer a simple and short proof based on an elementary
zero counting argument.

\begin{proof}
Write
$$P(e^{it}) = R(t)e^{i\varphi(t)}\,, \qquad R(t) := |P(e^{it})|\,,$$
where $R$ and $\varphi$ are differentiable functions on
$$[0,2\pi) \setminus \{t_1,t_2, \ldots, t_n\}\,.$$
Since $P^{\prime}(e^{it})$ is a continuous function of $t \in {\Bbb R}$,
in the rest of the proof we may assume that
$$t \in [0,2\pi) \setminus \{t_1,t_2, \ldots, t_n\}\,.$$
We have
$$P^{\prime}(e^{it})e^{it}i =
R^\prime(t)e^{i\varphi(t)} + iR(t)e^{i\varphi(t)}\varphi^\prime(t)\,,$$
and it follows that
\begin{equation}|P^{\prime}(e^{it})|^2 = (R^\prime(t))^2 + R(t)^2(\varphi^\prime(t))^2\, . \label{1.1}\end{equation}
Using the fact that  $w=z/(z-z_j)$ maps $\mathbb{T}$ onto the vertical line ${\rm Re}\, w=1/2,$ we have

$$ \varphi^{\prime}(t) =
\text{\rm Re} \left( \varphi^{\prime}(t) - \frac{R^{\prime}(t)}{R(t)} \,i \right)
= \text{\rm Re} \left( \frac {P^{\prime}(e^{it})e^{it}}{P(e^{it})} \right)
= \text{\rm Re} \left( \sum_{j=1}^n{\frac{e^{it}}{e^{it} - e^{it_j}}} \right)= \frac n2 \,.
$$
Thus from  \eqref{1.1} we get
\begin{equation}|P^{\prime}(e^{it})|^2 = (R^\prime(t))^2 + \left( \frac n2 \right)^2R(t)^2\,, \label{1.2}\end{equation}
and so, if $R(t) \geq m,$ then \eqref{main} follows immediately.\

Assume now that
$R(t) < m$. Observe that $Q$ defined by $Q(t) := R(t)^2 = |P(e^{it})|^2$
is a real trigonometric polynomial of degree $n$; that is, $Q \in {\mathcal T}_n$.
Now let
$$t_0 \in [0,2\pi) \setminus \{t_1,t_2, \ldots, t_n\}$$
be fixed, and let $T \in {\mathcal T}_n$ be defined by
$$T(t) := m^2 \cos^2(n(t-\alpha)/2) = \frac{m^2}{2}\,(1 + \cos(n(t-\alpha))\,,$$
where
$\alpha \in {\Bbb R}$ is chosen so that
\begin{equation}T(t_0) = Q(t_0) \label{1.3}\end{equation}
and
\begin{equation}\text{\rm sign}(T^{\prime}(t_0)) = \text{\rm sign}(Q^{\prime}(t_0))\,. \label{1.4}\end{equation}
We claim that
\begin{equation}|Q^{\prime}(t_0)| \geq |T^{\prime}(t_0)|\,. \label{1.5}\end{equation}
Indeed, $|Q^{\prime}(t_0)| < T^{\prime}(t_0)|$ together with \eqref{1.3} and \eqref{1.4} would imply
that the not identically zero trigonometric polynomial $T-Q \in {\mathcal T}_n$ had at least
$2n+~2$ zeros in the period $[0,2\pi)$ (at least two zeros on each of the intervals
$[t_1,t_2)$, $[t_2,t_3)$,$\ldots$, $[t_n,t_{n+1})$, and at least four zeros on the
interval $(t_j,t_{j+1})$ containing $t_0$) by counting multiplicities, a contradiction. Thus
 \eqref{1.5} holds and implies that
$$|Q^{\prime}(t_0)| \geq |m^2 n \cos(n(t_0-\alpha)/2)\sin(n(t_0-\alpha)/2)|\,,$$
which, together with \eqref{1.3}, yields
\begin{equation*}\begin{split} |Q^{\prime}(t_0)|^2& \geq  n^2(m^2 \cos^2(n(t_0-\alpha)/2))(m^2\sin^2(n(t_0-\alpha)/2)) \\
& = n^2 (|Q(t_0)|(m^2 - |Q(t_0)|)\,. \end{split}\end{equation*}
Substituting $Q(t_0) = R(t_0)^2$ and $Q^{\prime}(t_0) = 2R(t_0)R^{\prime}(t_0)$
in the above inequality, we conclude that
$$(R^{\prime}(t_0))^2 \geq \frac {n^2}{4}(m^2 - R(t_0)^2)\,. \label{1.6}$$
Finally, combining this last inequality with \eqref{1.2}   and recalling  that $R(t_0) = |P(e^{it_0})|$ yields \eqref{main}.
\end{proof}

 A natural question that arises is finding the maximal value $  m^*(n)$ of the quantity $m$ given in \eqref{mdef} or, equivalently (using the notation of Theorem~\ref{Thm1.1}),  determining
\begin{equation}m^*(n):=\max_{\omega_n \in \Omega_n} \min_{1 \leq j \leq n} \max_{t \in [t_j,t_{j+1}]}{\left|\prod_{k=1}^n{(e^{it}-z_k)}\right|\,}, \, \qquad z_k=e^{it_k},\label{n.3}\end{equation}
where $\Omega_n$ is the collection of all $n$-tuples
$\omega_n \in [0, 2\pi)^n$ of the form
$$\omega_n = (t_1,\ldots,t_n), \qquad
0 \leq t_1 \leq \ldots \leq t_n < 2\pi.$$
 In Corollary 6.9 of \cite{K-09}, Khrushchev proved that $m^*(n)=2,$ the value of $m$ corresponding to $P(z)=z^n-1$ for which equality holds throughout in \eqref{main}. Here
 we deduce this fact as a simple consequence of Theorem \ref{Thm1.1}.

\begin{corollary}\label{cor2}
Let $m^*(n)$ be as in \eqref{n.3}.  Then $m^*(n)=2$ and this  maximum is attained only for $n$ distinct equally spaced points $\{z_1,\ldots, z_n \}$ on the unit circle.
\end{corollary}

In other words, for any monic polynomial of degree $n$ all of whose zeros lie on the unit circle, there must be some sub-arc formed from consecutive zeros on which the modulus of the polynomial is at most 2.
\begin{proof}[Proof of Corollary~\ref{cor2}]
Assume $m^*(n)>2$.
According  to Theorem~\ref{Thm1.1}, $|P'_o(z)| >n$ for all $z$ on $\mathbb{T}$,
 where $P_o$ is a  monic polynomial of degree $n$ for which the maximum value $m^*(n)$ is attained.
By the Gauss-Lucas theorem, $P'_o$ has all its zeros in the open unit disk  (clearly it can't
have any on $\mathbb{T}$). So now consider the $f(z):=P'_o(z)/ {z^{n-1}}$, which is analytic
on and outside $\mathbb{T}$, even at infinity where it equals $n.$ Since
$f$ does not vanish outside or on $\mathbb{T},$ its modulus must attain its minimum
on $\mathbb{T}$.  But $|f(z)|>n$ on $\mathbb{T}$, while $f(\infty)=n$, which gives the desired contradiction.  Thus
$m^*(n)=2$ and the argument above also shows that if this maximum is attained by a polynomial
$P_o$, then $|f(z)|=2$ for all $z$ on or outside $\mathbb{T}$, which implies that $f$ is constant and so $P_o$ has equally
spaced zeros on $\mathbb{T}$.
\end{proof}
 Observe that the determination of $m^*(n)$ can alternatively be viewed as a min-max-min problem on the unit circle for the logarithmic potential $\log(1/r)$ with $r$ denoting Euclidean distance between points on $\mathbb{T}$.  In the next section we consider such problems for a general class of potentials.

\section{Min-max-min problems on $\mathbb{T}$}

Let $g$ be a positive, extended real-valued, even function defined on ${\Bbb R}$ that is periodic
with period $2\pi$ and satisfies $g(0)=\lim_{t\to 0}g(t)$.  Further suppose that $g$ is non-increasing and strictly convex on $(0,\pi]$.  For  $\omega_n =(t_1,\ldots,t_n)\in \Omega_n$, we set
\begin{equation}\label{potent} P_{\omega_n}(t) := \sum_{j=1}^n{g(t-t_j)}.
\end{equation}
Here and in the following we assume that
$t_j$ is extended so that $$t_{j+n}=t_j+2\pi, \qquad (j\in\mathbb{Z});$$ in particular, we have $t_0=t_n-2\pi$ and $t_{n+1}=t_1+2\pi$.
For $\omega_n=(t_1,\ldots,t_n)\in\Omega_n$ and $\gamma\in[0,2\pi)$,  let $\omega_n+\gamma$ denote the element in $\Omega_n$ corresponding to the set $\{e^{i(t_k+\gamma)}\}_{k=1}^n$.
Then $P_{\omega_n+\gamma}(t)=P_{\omega_n}(t-\gamma)$.
We further
let
$\widetilde{\omega}_n :=(\widetilde{t}_1,\ldots,\widetilde{t}_n)$
denote the equally-spaced configuration given by
$$\widetilde{t}_j := 2(j-1)\pi/n\,, \qquad j=1,2,\ldots,n\,.$$
By the convexity of $g$, it follows that
$$\min_{t\in[0,2\pi)} P_{\widetilde{\omega}_n}(t) = P_{\widetilde{\omega}_n}(\pi/n)\,.$$

Motivated by recent articles on polarization of discrete potentials on the unit circle (cf. \cite{A-09}, \cite{ABE-12}, \cite{ES-13}, \cite{HKS-12}) we shall prove the following generalization of Corollary~\ref{cor2}.

\begin{theorem}\label{Thm2.1}
Let $g$ be a positive, extended real-valued, even function defined on ${\Bbb R}$ that is periodic
with period $2\pi$ and satisfies $g(0)=\lim_{t\to 0}g(t)$.   Suppose further that $g$ is non-increasing and strictly convex on $(0,\pi]$. Then we have
\begin{equation}\min_{\omega_n \in \Omega_n}{\left\{\max_{1 \leq j \leq n}{\left\{ \min_{t \in [t_j,t_{j+1}]}{P_{\omega_n}(t)} \right\}}\right\}} = P_{\widetilde{\omega}_n}(\pi/n)\,;
\end{equation}
that is, the solution to the min-max-min problem on $\mathbb{T}$ is given by $n$ distinct  equally-spaced  points on $\mathbb{T}$ and, moreover, these are the only solutions.
\end{theorem}

\subsection{Logarithmic and Riesz kernels\label{example}} It is straightforward to verify that $g(t)=g_{\log}(t):=\log(1/|e^{it}-1|)=-\log(2 \sin |t/2|)$   satisfies the hypotheses of Theorem~\ref{Thm2.1} providing an  alternate proof of Corollary~\ref{cor2} .  Furthermore, for the case (relating to Euclidean distance),
 \begin{equation}\label{fs}
g(t)=g_s(t):=  |e^{it}-1|^{-s}=(2 \sin |t/2|)^{-s}, \, s>0,
 \end{equation}
we obtain the {\em Riesz $s$-potential} and it is again easily verified that $g_s$ satisfies the hypotheses of
Theorem~\ref{Thm2.1}.  Consequently, with $z_k=e^{it_k},$\\
\begin{equation}\label{n.4} \min_{\omega_n \in \Omega_n} \max_{1 \leq j \leq n} \min_{t \in [t_j,t_{j+1}]}\sum_{k=1}^n{|e^{it}-z_k|^{-s}}=\sum_{k=1}^n|e^{i\pi /n}-e^{2k\pi i/n}|^{-s}= M^s_n(\mathbb{T}),
\end{equation}
where $M^s_n(\mathbb{T})$ is the \textit{Riesz $s$-polarization constant} for $n$ points on the unit circle (cf. \cite{HKS-12}).  We remark that
for $s$ an even
integer, say $s=2m$, the precise value of $M^{s}
_n(\mathbb{T})$ can be expressed as a polynomial in $n$; namely, as a consequence of the formulas derived in \cite{BHS-09},
\begin{equation}
\label{Mn}
M_n^{2m}(\mathbb{T})= \frac{2}{(2\pi)^{2m}}
\sum_{k=1}^m n^{2k}\zeta(2k)\alpha_{m-k}(2m)(2^{2k}-1),\quad m \in \mathbb{N}, \end{equation}
where $\zeta(s)$ is the classical Riemann zeta function and $\alpha_j(s)$ is defined via the power series for
$\text{\rm sinc\,} z = (\sin \pi z)/(\pi z):$
$$(\text{\rm sinc\,} z)^{-s} = \sum_{j=0}^\infty \alpha_j(s) z^{2j}\,;
\quad  \alpha_0(s)=1\,,$$
see Corollary~3 from \cite{HKS-12}.
In particular,
\begin{equation*}
M^2_n(\mathbb{T})= \frac{n^2}{4},\quad
 M_n^4(\mathbb{T})  =
 \frac{n^2}{24}+\frac{n^4}{48} ,\quad
 M_n^6(\mathbb{T})  = \frac {n^2} {120} + \frac {n^4} {192} + \frac {n^6} {480}.
\end{equation*}

\subsection{Proof of Theorem \ref{Thm2.1}. \label{proof}} Theorem~\ref{Thm2.1} is a consequence of the following lemma which is the basis of the proof of the polarization theorem established by Hardin, Kendall and Saff in \cite{HKS-12}. (In Section 3 we state a slightly stronger
version of this polarization result as Theorem~\ref{Thm3.1} and present some related results.)

\begin{lemma}\label{Lemma2.2}
Let $g$ be as in Theorem~\ref{Thm2.1} and suppose $\omega_n=( t_1, \ldots, t_n)$ and $\omega'_n=( t'_1, \ldots, t'_n)$ are   in $\Omega_n$.
Then there is some $\ell\in \{0,1,\ldots,n\}$ and some $\gamma \in [0,2\pi)$ (where $\ell$ and $\gamma$ depend on $\omega_n$ and $\omega'_n$ but not on $g$) such that
\begin{equation}
\label{eqn2.2}P_{\omega_n'}(t-\gamma) \leq P_{ \omega_n}(t)\,, \qquad t \in [t_\ell, t_{\ell+1}],
\end{equation}  and  $ [t_{\ell},t_{\ell+1}] \subset [t'_{\ell}+\gamma,t'_{\ell+1}+\gamma]$.

The inequality is strict for $t\in(t_\ell, t_{\ell+1})$ unless  $t_{j+1}-t_j=t'_{j+1}-t'_j$ for all
$j=1,\ldots,n$.
\end{lemma}

\begin{proof}[Sketch of proof]
This lemma follows from techniques developed in  \cite{HKS-12}, specifically from Lemmas 5 and 6 in that paper.
For the convenience of the reader, we  provide here an outline of its  proof.   First, the convexity of $g$ implies   that,
for $n=2$, the inequality
\begin{equation}\label{twopts}
P_{(t_1-\Delta,\, t_2+\Delta)}(t)<P_{(t_1 ,\, t_2 )}(t), \qquad t\in (t_1,t_2),
\end{equation}
holds for sufficiently small $\Delta>0$ (this observation was also used in \cite{ABE-12}).   That is, the potential due to two points decreases on an interval when the points are moved symmetrically {\em away } from the interval.   For simplicity, we consider the case that $$\text{sep}(\omega_n):=\min_j (t_{j+1}-t_j)>0,$$  (see \cite{HKS-12} for the case of coincident points where $\text{sep}(\omega_n)=0$).

Next, using elementary linear algebra, we find a   vector $\boldsymbol{\Delta}=(\Delta_1,\ldots,\Delta_n)$ such that (a) $\Delta_k\ge 0$ for  all $k$, (b)  $\Delta_{\ell}=0$ for some $\ell$ and (c) $\boldsymbol{\Delta}$ solves the equations
\begin{equation}\label{Delta}
(t_{j+1}'-t_{j}')=(t_{j+1}-t_j)-\Delta_{j+1}+2\Delta_j -\Delta_{j-1}, \qquad (j=1,\dots, n)
\end{equation}
where we take $\Delta_{0}:=\Delta_n$ and $\Delta_{n+1}:=\Delta_1$.   For $j=1, \ldots, n$,  consider the transformation $$\tau_{j,\Delta}(\omega_n):=(t_1,\ldots,t_{j-2},t_{j-1}-\Delta,t_j+\Delta, t_{j+1},\ldots,t_n).$$
Then \eqref{Delta} implies that $\omega''_n:=\tau_{1,\Delta_1}\circ\tau_{2,\Delta_2}\circ\cdots\circ\tau_{n,\Delta_n}(\omega_n)$ equals
$\omega_n'+\gamma$ for some $\gamma\in[0,2\pi)$.   If $\max_j\Delta_j\le (1/2) \text{sep}(\omega_n)$ then, since $\Delta_\ell=0$ and
$\Delta_k\ge0$, we may  apply the inequality  \eqref{twopts} $n$ times to obtain
\eqref{eqn2.2}. Moreover, unless $\Delta_k=0$ for all $k$,  inequality \eqref{Delta} is strict.   If $\max_j\Delta_j> (1/2) \text{sep}(\omega_n),$ then we may choose $m$ such that  $(1/m)\max_j\Delta_j< (1/2) \text{sep}(\omega_n)$  and then recursively applying $\tau_{(1/m) \boldsymbol{\Delta}}$ to $\omega_n$ the number $m$ times,  we again obtain \eqref{eqn2.2}.

Finally, since $\Delta_{\ell}=0$ and $\Delta_{\ell-1},\Delta_{\ell+1}\ge 0$, we have $[t_{\ell},t_{\ell+1}]\subset [t''_{\ell},t''_{\ell+1}]  = [t'_{\ell}+\gamma,t'_{\ell+1}+\gamma]$.
\end{proof}

\begin{proof}[Proof of Theorem~\ref{Thm2.1}]

Let $\omega_n\in\Omega_n$ be fixed but arbitrary and recall that $\widetilde{\omega}_n$ denotes an equally spaced configuration.  By Lemma~\ref{Lemma2.2}, there is some  $\ell\in \{0,1,\ldots,n\}$ and some $\gamma \in [0,2\pi)$ such that
$$P_{\widetilde{\omega}_n}(t-\gamma) \leq P_{ \omega_n}(t)\,, \qquad t \in [t_\ell, t_{\ell+1}].$$
Hence,
\begin{equation*}
\begin{split}
P_{\widetilde{\omega}_n}(\pi/n)& = \min_{t\in [0,2\pi)}P_{\widetilde{\omega}_n}(t)\le   \min_{t \in [t_\ell, t_{\ell+1}]}P_{\widetilde{\omega}_n}(t-\gamma)\\
& \le \min_{t \in [t_\ell,t_{\ell+1}]}P_{{\omega}_n}(t)
\le \max_j\min_{t \in [t_j,t_{j+1}]}P_{{\omega}_n}(t).
\end{split}
\end{equation*}
\end{proof}

\subsection{Derivatives of logarithmic potentials}
We next consider a class of kernels $g$ derived from $g_{\log}$ that were considered in \cite{ES-13}.    For an even positive integer $m$, we define the kernel:
$$g_m(t):=g_{\log}^{(m)}(t)=\frac{d^m}{dt^m}g_{\log}(t).$$
Then, for $t\in[0,2\pi)$,
$$g_2(t) =
\frac{d}{dt} \left( -\frac 12 {\cot \left( \frac{t}{2} \right)} \right)
= \frac 14 {\csc^2 \left( \frac{t}{2} \right)}\,$$
and hence
$$g_m(t) = \frac 14 {f^{(m-2)}(t)} \,,$$
where $f(t) := \csc^2(t/2)$.  Following  \cite{ES-13}, we next verify that $g_m$ satisfies the hypotheses of Theorem~\ref{Thm2.1}.
It is well known and elementary to check that
$$\tan t = \sum_{j=1}^{\infty}{a_jt^j}\,, \qquad t \in (-\pi/2,\pi/2)\,,$$
with each $a_j \geq 0$, $j=1, 2,\ldots$. Hence, if $h(t) := \tan(t/2)$, then
$$h^{(k)}(t) > 0, \qquad t \in (0,\pi), \qquad k=0,1,\ldots\,.$$
Now observe that
$$f(t) = \csc^2 \left( \frac t2 \right) = \sec^2{\frac{\pi - t}{2}} = 2h^\prime(\pi - t)\,,$$
and hence,
$$(-1)^kf^{(k)}(t) = 2h^{(k+1)}(\pi - t) > 0, \qquad t \in (0,\pi)\,.$$
This implies that if $m$ is even, then $g_m(t) = \frac 14 f^{(m-2)}(t)$ is a positive, decreasing, strictly convex function on $(0,\pi)$.
It is also clear that if $m$ is even, then $g_m$ is even since $f$ is even. Thus, $g=g_m$ satisfies the hypotheses of Theorem~\ref{Thm2.1}.

We remark that, for  an even positive integer $m$, an induction argument implies that
$$g_m(t)=p_m(r^{-2}), \qquad r=2\sin(t/2),$$
for some polynomial $p_m$ of degree $m/2$.  The induction follows from the recursive relation
\begin{equation}\label{Pm}
p_{m+2}(x)=(6 x^2 - x)p_{m}'(x)+(4 x^3 - x^2)p_{m}''(x),
\end{equation}
which is easily derived using $(r')^2=1-(r/2)^2$ and $r''=-r/4$.   Thus, $g_m$ can be expressed as a linear combination of Riesz $s$-potentials with $s=2,4,\ldots, m$ with coefficients corresponding to the polynomial $p_m$.   Table~\ref{fig1}
displays $p_m$ for $m=2,4,6,$ and $8$.

For $\omega_n \in\Omega_n$, we let
$$Q_{\omega_n}(t) := \prod_{j=1}^n{\sin \left| \frac{t-t_j}{2} \right|}\,$$
and  set
$$T_n(t) := Q_{\widetilde{\omega}_n}(t) = 2^{1-n} \sin \left|\frac{nt}{2}\right|\,.$$

Our next two results are consequences of Lemma~\ref{Lemma2.2} and Theorem~\ref{Thm2.1}, respectively.

\begin{theorem}\label{Thm2.3} Let $m$ be a positive even integer and $\omega_n\in \Omega_n$.  Then there is some $\gamma \in [0,2\pi)$ and some $j\in\{1,2,\ldots, n\}$ (with $\gamma$ and $j$ depending on $\omega_n$) such that
$$-(\log|Q_{\omega_n}|)^{(m)}(t) \geq  -(\log|T_n|)^{(m)}(t-\gamma)\,, \qquad t \in (t_j,t_{j+1})\,.$$
\end{theorem}
\begin{proof}
This is an immediate consequence of   Lemma~\ref{Lemma2.2} with $g=g_m$ and $\omega'_n=\widetilde{\omega}_n$, and so $P_{\omega_n}(t)=-(\log|Q_{\omega_n}|)^{(m)}(t)$ and $P_{\widetilde{\omega}_n}(t)=-(\log|T_n|)^{(m)}(t)$).
\end{proof}

Since $g_m$ satisfies the hypotheses of  Theorem~\ref{Thm2.1}, we obtain the following theorem.

\begin{theorem}\label{Thm2.4}
We have
\begin{equation*}
\min_{\omega_n \in \Omega_n}{\left\{\max_{1 \leq j \leq n}{\left\{ \min_{t \in [t_j,t_{j+1}]}{-(\log|Q_{\omega_n}|)^{(m)}(t)} \right\}}\right\}} =-(\log|T_n|)^{(m)}(\pi/n)\,,
 \end{equation*}
for every even positive integer $m$.
\end{theorem}

From \eqref{Pm}, one can show that the leading coefficient of $p_m$ is $(m-1)!$.  A somewhat more detailed computation using \eqref{Mn} and \eqref{Pm} yields
\begin{equation}\label{Tn}
-(\log|T_n|)^{(m)}(\pi/n)=\frac{2}{(2\pi)^{m}}\zeta(m)(m-1)!(2^m-1).
\end{equation}
Table~\ref{fig1} gives  the values  $
-(\log|T_n|)^{(m)}(\pi/n)$ for $m=2, 4, 6, 8,$ and for $n\in {\mathbb N}$.
\bigskip

\begin{table}[bth]
\begin{center}
\begin{tabular}{|c|c|c|}
\hline
$m$ & $p_m(x)$ & $-(\log|T_n|)^{(m)}(\pi/n)$\\
\hline
2 & $x$& $n^2/4$\\ \hline
4 & $6x^2-x$& $n^4/8$\\ \hline
6 & $120x^3-30x^2+x$& $n^6/4$\\ \hline
8 & $5040x^4-1680x^3+126x^2-x$& $17n^8/16$\\ \hline
\end{tabular}
\end{center}
\caption{The polynomials $p_m(x)$ and the values  $
-(\log|T_n|)^{(m)}(\pi/n)$ from \eqref{Tn} (see Theorem~\ref{Thm2.4} and Corollary~\ref{Cor3.3}) for $m=2, 4, 6, 8,$ and for $n\in {\mathbb N}$.} \label{fig1}
\end{table}

\section{Comments on polarization}
The main part of the following `polarization' theorem was proved in \cite{HKS-12}.   As observed in \cite{ES-13}, for each $\omega_n\in\Omega_n$, we may restrict the set over which we search for a minimum  to
 $$E(\omega_n) := [0,2\pi) \setminus \bigcup_{j=1}^n{ \left(t_j - \pi/n, t_j + \pi/n \right)} \quad (\text {\rm mod} \enskip 2\pi)\,.$$

\begin{theorem}\label{Thm3.1}
Let $g$ be as in Theorem~\ref{Thm2.1}. Then
\begin{equation}\label{polar}\max_{\omega_n \in \Omega_n}\left\{ \min_{t \in [0,2\pi)}{P_{\omega_n}(t)} \right\} =\max_{\omega_n \in \Omega_n}\left\{ \min_{t \in E(\omega_n)}{P_{\omega_n}(t)} \right\} = P_{\widetilde{\omega}_n}(\pi/n)\,.
\end{equation}
\end{theorem}
\begin{proof}  Let $\omega_n\in \Omega_n$ be  arbitrary.  The proof follows from Lemma~\ref{Lemma2.2} and is similar to the proof of Theorem~\ref{Thm2.1}, except that the roles of $\widetilde{\omega}_n$ and $\omega_n$ are switched.
 By Lemma~\ref{Lemma2.2}, there is some  $\ell\in \{0,1,\ldots,n\}$ and some $\gamma \in [0,2\pi)$ such that
$$P_{ \omega_n}(t-\gamma) \leq P_{\widetilde{\omega}_n}(t)\,, \qquad t \in [\widetilde{t}_\ell, \widetilde{t}_{\ell+1}],$$
and $[\widetilde{t}_{\ell},\widetilde{t}_{\ell+1}] \subset [t_{\ell}+\gamma,t_{\ell+1}+\gamma]$.  Then $$t_{\ell}+\pi/n\le \pi(2\ell-1)/n-\gamma\le t_{\ell+1}-\pi/n,$$ and so  $\pi(2\ell+1)/n-\gamma\in E(\omega_n)$.
We then obtain
\begin{equation*}
\begin{split}
 \min_{t \in [0,2\pi)}P_{{\omega}_n}(t)&\le\min_{t \in E(\omega_n)}P_{{\omega}_n}(t) \le P_{{\omega}_n}(\pi(2\ell+1)/n-\gamma)\\
 &\le P_{\widetilde{\omega}_n}(\pi(2\ell+1)/n)= P_{\widetilde{\omega}_n}(\pi/n),
\end{split}
\end{equation*}
which completes the proof.
\end{proof}

\begin{theorem}\label{Thm3.2} Let $\omega_n \in\Omega_n$.
Then there is a number $\theta \in [0,2\pi)$ (depending on $\omega_n$) such that
$$-(\log|Q_{\omega_n}|)^{(m)}(t) \leq  -(\log|T_n|)^{(m)}(t-\theta)\,, \qquad t \in (\theta,\theta + 2\pi/n)\,,$$
for every nonnegative even integer $m$.
\end{theorem}
\begin{proof}
Let $m$ be a nonnegative even integer.
We apply Lemma~\ref{Lemma2.2} with $g=g_m$, $\omega_n'=\omega_n$ and $\omega_n=\widetilde{\omega}_n$ (in which case, $P_{\omega_n}(t)=-(\log|Q_{\omega_n}|)^{(m)}(t)$ and $P_{\widetilde{\omega}_n}(t)=-(\log|T_n|)^{(m)}(t)$) to deduce that there is an $\ell \in \{1,2,\ldots,n\}$ and a number $\gamma \in [0,2\pi)$
(depending  on $\omega_n$) such that
$$-(\log |Q_{\omega_n}|)^{(m)}(t-\gamma)  \leq -(\log|T_n|)^{(m)}(t)\,, \qquad t \in [\widetilde{t}_{\ell},\widetilde{t}_{\ell+1})\,,$$ which can be rewritten using
$\theta:=\widetilde{t}_{\ell}-\gamma$, $u:=t-\gamma$, and the fact that $T_n$ is $2\pi/n$ periodic as
$$-(\log |Q_{\omega_n}|)^{(m)}(u)  \leq -(\log|T_n|)^{(m)}(u-\theta)\,, \qquad u \in [\theta,\theta+2\pi/n)\,.$$
\end{proof}

\begin{corollary}\label{Cor3.3}
We have
\begin{equation*}
 \begin{split}
\max_{\omega_n \in \Omega_n} \left\{ \min_{t \in [0,2\pi)}{-(\log|Q_{\omega_n}|)^{(m)}(t)} \right\}&=\max_{\omega_n \in \Omega_n} \left\{ \min_{t \in E(\omega_n)}{-(\log|Q_{\omega_n}|)^{(m)}(t)} \right\}\\
&= -(\log|T_n|)^{(m)}(\pi/n) \end{split}
 \end{equation*}
for every even integer $m$.
\end{corollary}
\begin{proof}
 This  is an immediate corollary of   Theorem~\ref{Thm3.1}.
\end{proof}



\begin{thebibliography}{99}

\bibitem{A-09} G. Ambrus,
{\em Analytic and Probabilistic Problems in Discrete Geometry},
  Ph.D. Thesis, University College London, 2009


\bibitem{ABE-12}{  G. Ambrus, K. Ball, and T. Erd\'elyi}, {\em Chebyshev constants for the unit circle},  Bull. London Math. Soc., {\bf 45}(2) (2013), 236--248.

\bibitem{B-26}{  S.N. Bernstein,}, {\em Le\c cons sur les propri\'et\'es extr\'emales et la meilleure approximation des
fonctions analytiques d'une variable r\'eelle},  Gauthier-Villars,   Paris  (1926).


\bibitem{BE-95}{   P. Borwein and T. Erd\'elyi},  {\em Polynomials and Polynomial Inequalities},  Springer-Verlag,   New York, N.Y.  (1995).




\bibitem{BHS-09}{  J.~S. Brauchart, D.~P. Hardin, and E.~B. Saff}, {\em The {R}iesz energy of the {$N$}th roots of unity: an asymptotic expansion for large {$N$}},  Bull. London Math. Soc. {\bf 41}(4)  (2009),  621--633.




\bibitem{DL-93}{ R.~A. DeVore and G.~G. Lorentz},  {\em Constructive Approximation},  Springer-Verlag, Berlin, Heidelberg  (1993).

\bibitem{ES-13} T. Erd\'elyi and E. B. Saff, {\em Riesz polarization in higher dimensions}, J. Approx. Theory,  {\bf 171} (2013), 128-147.







\bibitem{HKS-12}{  D.~P. Hardin, A. P. Kendall and E.~B. Saff}, {\em Polarization optimality of equally spaced points on the circle for discrete potentials},  (to appear) Discrete Comput. Geom.


\bibitem{K-09}{  S. Khrushchev}, {\em Rational compacts and exposed quadratic irrationalities}, J. Approx. Theory {\bf 159} (2) (2009), 243-289.
















 \end{thebibliography}
\end{document}